\documentclass[12pt]{amsart}
\usepackage{amsmath, amssymb, amstext, amsthm, eqnarray, amscd}
\usepackage{tikz}
\usepackage{tikz-cd}
\usetikzlibrary{%
  matrix,%
  calc,%
  arrows%
}
\usepackage{fullpage}

\usepackage{float}
\usepackage{mathtools}
\usepackage{mathabx}
\usepackage{graphicx}
\usepackage[mathscr]{euscript}
\usepackage{enumerate}
\usepackage{hyperref}
\usepackage{pgfkeys}
\usepackage{fullpage}
\usepackage{pst-node}
\usepackage{lastpage}
\usepackage{fancyhdr}
\usepackage{multirow}
\allowdisplaybreaks
\usepackage{graphicx}
\usepackage{tikz,color,soul}
\usepackage{chngcntr}
\usepackage{stmaryrd}
\usepackage{tabularx}

\usepackage{enumitem}
\usepackage{kantlipsum}
\usepackage{array}
\usepackage{cleveref}
\usepackage[english]{babel}
\usepackage{enumitem}
\def\multiset#1#2{\ensuremath{\left(\kern-.3em\left(\genfrac{}{}{0pt}{}{#1}{#2}\right)\kern-.3em\right)}}

\newtheorem{theorem}{Theorem}[section]

\newtheorem{cor}[theorem]{Corollary}
\theoremstyle{definition}
\newtheorem{definition}[theorem]{Definition}
\newtheorem{eg}[theorem]{Example}
\newtheorem{remark}[theorem]{Remark}

\numberwithin{equation}{section}

\newcommand{\nothing}[1][]{%
    \tikz[overlay,remember picture]{
    \draw[red]
      ($(l4)+(-0.7em,0.7em)$) rectangle
      ($(r4)+(0.5em,-0.5em)$);}
}

\title{On Reversible rings, their generalization and quasi-idempotent elements}

\author{Vivek Bhabani Lama}
   \address{Department of Mathematics, Indian Institute of Technology Kharagpur, Kharagpur, INDIA - 721302.}
   \email{ v.bhabani.lama@gmail.com, vivekbhabanilama@kgpian.iitkgp.ac.in}


\date{}

\makeindex
\begin{document}

\maketitle


\section*{Abstract}
In this article, we show that the characterization of reversible rings and various other rings related to them such as i-reversible rings, semi-commutative rings and i-domains in terms of properties related to idempotent elements known in the literature can be generalized to exactly the same properties for quasi-idempotent elements. 

\section{Introduction}

For an associative ring, an element $r\in R$ is called an idempotent if $r^2=r$. Idempotent elements encode the structural information patterns in various classes of associative rings. Recently, Su et. al \cite{Liu2026} introduced the notion of a quasi-idempotent element in an associative ring. This ser of elements may be much larger than the aforementioned set of idempotent elements in a ring. 

Reversible rings were formally introduced by Cohn in 1999 (\cite{cohn1999reversible}). A ring $R$ is said to be reversible if for every $a,b\in R$, $ab=0$ then $ba=0$. Anderson and Camillo  \cite{anderson1999semigroups}, called this class of rings as $ZC_2$, while Krempa and Niewieczerzal \cite{krempa1977rings}, called them $C_0$ rings.

One of the most well celebrated characterization of reversible rings is that: A ring $R$ is reversible if and only if for every $a,b\in R$, $ab$ is an idempotent implies $ba$ is an idempotent. In our first main theorem, we extend this to the setting of quasi-idempotent elements. More precisely, we prove:
\vskip 2mm
\noindent
	{\bf Theorem A \bf (Theorem \ref{main 1}).}
 \textit{  A ring $R$ is reversible if and only if for every $a,b \in R$ such that $ab$ is a quasi-idempotent implies that $ba$ is also a quasi-idempotent.}
\vspace{2mm}

Anjana Khurana and Dinesh Khurana \cite{khurana2020reversible}, studied a new class of rings where $ab$ is a non-zero idempotent implies $ba$ is an idempotent and called them i-reversible rings. It was shown that every reversible ring is i-reversible. 
  Jung et.al \cite{jung2019reversibility}  studied the same class of rings under the terminology of quasi-reversible rings.
  In our second main theorem, we show that not only the characterization for reversible rings extends fo quasi-idempotent elements but replacing idempotents by quasi-idempotents in the definition of i-reversible rings would give rise to the same class of rings. The precise statement is as follows:
  \vskip 2mm
\noindent
	{\bf Theorem B \bf (Theorem \ref{main 2}).}
 \textit{A ring $R$ is i-reversible if and only if for every $a,b \in R$ such that $ab$ is a non-zero quasi-idempotent implies that $ba$ is also an quasi-idempotent. }
\vspace{2mm}

Finally as our last theorems, we extend characterizations of semi-commutative rings and i-domains to quasi-idempotent elements. These classes of rings are very closely related to the class of reversible rings.

\section{Reversible Rings, Other Related Rings and Quasi-reversible elements.}\label{Section S2}
 In this section, we recall the definitions of various rings that are related to reversible rings. We also provide relevant context on quasi-reversible elements of a ring. We prove that each of these rings are characterized in terms of certain properties  of quasi-idempotent elements. In particular these properties are the same ones that characterize these rings in terms of their idempotent elements.
 
 From now on, by a ring, we mean an associative ring with unity (not necessarily commutative) unless stated otherwise. We start the section with the definition of quasi-idempotent elements and its relationship with idempotent elements of the ring 
\begin{definition}\label{DEF: quasi idempotent}
Let $R$ be a ring, an element $e \in R$ is said to be a quasi-idempotent element of $R$ if $e^2=ke$ where $k$  is a central unit in $R$. 

\end{definition}
\begin{remark} \label{remark on anti-commutator}
    It is clear that every idempotent element is a quasi idempotent element with $k=1.$ The trivial idempotents $0,1$ are also trivial quasi-idempotents in the ring $R$.
\end{remark}

We now provide an example to show that every quasi-idempotent element need not be an idempotent.

\begin{eg}
Let $R$ be a ring and $e$ be an idempotent in $R$. Also let $k$ be a central unit in $R$ such that $e\neq ke$. Then, $x=ek$ is a quasi-idmpotent which is not an idempotent.
\end{eg}

We now prove the characterization of reversible rings in terms of quasi-idempotent elements. In order to do this, we provide context and definitions, wherever necessary.

\begin{definition}
    A ring $R$ is said to be reversible if for every $a,b \in R$ such that $ab=0$ implies $ba=0$.
\end{definition}

The following theorem is one of the most well celebrated characterization of reversible rings is in terms of its idempotent elements (see  for details). 

\begin{theorem}\label{Char of rev rings}\cite[Proposition 1.4.]{Jung2015}
    A ring $R$ is reversible if and only if for every $a,b \in R$ such that $ab$ is an idempotent implies that $ba$ is also an idempotent.
\end{theorem}

We now characterize this property of reversibility in terms of quasi-idempotent elements. What turns out to be remarkable is that the same property for quasi-idempotent elements as in \Cref{Char of rev rings} is exactly required to capture the reversibility of the ring $R$.

\begin{theorem} \label{main 1}
    A ring $R$ is reversible if and only if for every $a,b \in R$ such that $ab$ is a quasi-idempotent implies that $ba$ is also a quasi-idempotent.
\end{theorem}
\begin{proof}
  Suppose that $R$ is a reversible ring and let $a,b \in R$ such that $ab$ is a quasi-idempotent element in $R$. Therefore, we have the following equations:
\begin{align*}
    &(ab)^2=k( ab) \text{ for some central unit } k \in R.\\
    &\implies ab(ab-k)=0.
\end{align*}
Using the fact that $R$ is reversible, we get $b(ab-k)a=0$ which can be further reduced to $baba-bka=0.$ Now we use the fact that the element $k$ is central in $R$ and hence $baba-kba=0$. This is exactly $(ba)^2=k(ba)$ and hence, $ba$ is a quasi-idempotent as required.
Conversely, suppose that the ring $R$ satisfies the property that for every $a,b \in R$ such that $ab$ is a quasi-idempotent implies that $ba$ is also a quasi-idempotent. Now consider $a,b\in R$ such that $ab=0$. Since $ab=0$ it is trivially a quasi-idempotent (see \ref{remark on anti-commutator}). Therefore, $ba$ is also a quasi-idempotent by our hypothesis. Hence, there exists a central unit $k \in R$ such that $(ba)^2=k(ba)$. However since $ab=0$, this reduces to $0=(ba)^2=k(ba)$. Hence $k(ba)=0$ gives $ba=0$ since $k$ is a unit in $R$. Therefore the ring $R$ is reversible as desired.
\end{proof}

An interesting class of rings that was introduced from the notion of reversible rings is the class of i-reversible rings. These rings were first introduced in \cite{khurana2020reversible} and independently in \cite{jung2019reversibility} under the nomenclature of quasi-reversible rings. We now show that the i-reversible property of a ring is captured by quasi-idempotent elements. This creates an analogous result to \Cref{main 1}. We first recall the definition of i-reversible rings and follow the nomenclature as i-reversible rings throughout the article.

\begin{definition}
     A ring $R$ is i-reversible if for every $a,b \in R$ such that $ab$ is a non-zero idempotent implies that $ba$ is also an idempotent.
\end{definition}
\begin{theorem}\label{main 2}
    A ring $R$ is i-reversible if and only if for every $a,b \in R$ such that $ab$ is a non-zero quasi-idempotent implies that $ba$ is also an quasi-idempotent. 
\end{theorem}
\begin{proof}
Suppose that $R$ is an i-reversible ring and let $a,b \in R$ such that $ab$ is a non-zero quasi-idempotent. This gives that $(ab)^2=k(ab)$ for some central unit $k \in R$. Set $a'=k^{-1}a\in R$ and consider the following equalities:
\begin{align*}
    (a'b)^2&=(k^{-1})^2(ab)^2\\
    &=(k^{-1})^2(kab)\\
    &=k^{-1}ab\\
    &=ab.
\end{align*}
Therefore, $a'b$ is an idempotent in $R$. Also, $a'b=k^{-1}ab$ is a non-zero idempotent in $R$ for if otherwise, $ab=0$ which is not true due to our hypothesis. Therefore, using the fact that our ring $R$ is an i-reversible ring, we have that $(ba')^2=ba$ which gives rise to the following equalities:
\begin{align*}
    (ba')^2=&( ba')\\
    \implies (bk^{-1}a)^2=& \ bk^{-1}a\\
    \implies k^{-1}(ba)^2=& \ ba
\end{align*}
This finally gives $(ba)^2=k(ba)$as required. Conversely, suppose that the ring $R$ satisfies the property that for every $a,b \in R$ such that $ab$ is a non-zero quasi-idempotent implies that $ba$ is also an quasi-idempotent. Consider $a,b \in R$ such that $ab$ is a non-zero idempotent. Since $ab$ is a non-zero idempotent, by Remark \ref{remark on anti-commutator} $ab$ is also a non-zero quasi-idempotent. Therefore, $(ba)^2=k(ba)$ for some central unit $k\in R.$ Now consider the following expression for the element $(ba)^3 \in R:$
\begin{align*}
    (ba)^3&=b(ab)^2a\\
    &=b(ab)a\\
    &=(ba)^2\\
    &=k^2(ba).
\end{align*}
On the other hand, we also have that $(ba)^3=(ba)(ba)^2=(ba)(k(ba))=k(ba)^2=k^2(ba).$ Combining the two equalities for the element $(ba)^3$ we get $k^2(ba)=k(ba)$ which reduces to $kba=ba$. Since $(ba)^2=kba$, we finally get $(ba)^2=ba$ thereby completing the proof.
\end{proof}

Another class of rings that are closely associated to reversible rings is the semi-commutative rings. A  characterization of semi-commutative rings via properties related to idempotents has been established recently in \cite{SaadZailaee2026}. We provide the necessary details about semi-commutative rings and then prove our next main result (\Cref{main 3}). also recovers the characterization in \cite{SaadZailaee2026}.

\begin{definition}
    A ring $R$ is called a semi-commutative ring if for every $a,b \in R$ such that $ab=0$, we have $aRb=0$.
\end{definition}
\begin{theorem}\label{main 3}
  A ring $R$  is semi-commutative if and only if for every $a,b \in R$ such that $ab$ is a quasi-idempotent with $(ab)^2=k(ab)$ for some central unit $k \in R$, we have that $(k-ab)aRb(k-ab)=0.$
\end{theorem}
\begin{proof}
    Suppose that $R$ is semi-commutative and let $a,b \in R$ such that $(ab)^2=k(ab)$ for some central unit $k \in R.$ This automatically gives $(k-ab)ab=0$. Using the semi-commutative property of $R$ we get $(k-ab)aRb=0$ and hence $(k-ab)aRb(k-ab)=0.$ Conversely, suppose that $ab=0$ then $(ab)^2=k(ab)$ for any non trivial central idempotent $k\in R$. We pick any such central unit $k$ and by our hypothesis we get $kaRbk=0$ but since $k$ is a unit in $R$, this reduces to $aRb=0$. Therefore $R$ is semi-commutative.
\end{proof}
The above theorem immediately recovers the characterization of semi-commutative rings in terms of idempotent elements.

\begin{cor}\cite[Theorem 1.]{SaadZailaee2026}
    A ring $R$  is semi-commutative if and only if for every $a,b \in R$ such that $ab$ is an idempotent, we have that $(k-ab)aRb(k-ab)=0.$
\end{cor}
\begin{proof}
    The proof is direct by setting $k=1$ in the proof of \Cref{main 3}.
 \end{proof}

Finally, we prove our last theorem regarding the characterization of i-domains, introduced in \cite{SaadZailaee2026}. These are generalizations of domains in terms of idempotent elements. We show that a ring being an i-domain is equivalent to it satisfying the same properties for quasi-idempotent elements.

 \begin{definition}
     A ring $R$ is called an i-domain if for every idempotent $ab \in R$ for some $a,b \in R$ we have either $a\in I(b)$ or $b \in I(a)$ where $I(a) $ and $I(b)$ represent the two sided ideals generated by $a$ and $b $ respectively.
 \end{definition}

 \begin{theorem}
     A ring $R$ is an i-domain if and only if for every quasi-idempotent $ab \in R$ for some $a,b \in R$ we have either $a\in I(b)$ or $b \in I(a)$.
 \end{theorem}
 \begin{proof}
     Suppose that $ab$ is a quasi-idempotent elements in $R$. There exists a central unit $k \in R$ such that $(ab)^2=kab$, we have that $(a'b)$ is an idempotent element, where $a'=k^{-1}a$. Using the fact that $R$ is an i-domain, either $a'\in I(b)$ or $b \in I(a')$ which in turn implies that either $a\in I(b)$ or $b \in I(a)$.
     The converse is obvious, considering the fact that every idempotent is a quasi-idempotent.
 \end{proof}
 \section*{Acknowledgments}
 The author is thankful to the Government of India for supporting him in this work through the Prime Minister Research Fellowship. GPT-5.6 Luna was used for valuable discussions throughout the preparation of this article.

\section*{Declarations}
\subsection*{Data Availability}
No data was used for the research described in this note.
 
 \bibliographystyle{abbrv}
\bibliography{refs}
\end{document}